\documentclass{amsart}
 \usepackage{amsthm,amsfonts,amsmath,amssymb,latexsym,epsfig,upref,eucal,ae}

\usepackage[all]{xy}
\usepackage{color}

\newtheorem{theorem}[subsubsection]{Theorem}
\newtheorem{lemma}[subsubsection]{Lemma}
\newtheorem{proposition}[subsubsection]{Proposition}
\newtheorem{corollary}[subsubsection]{Corollary}

\newenvironment{remark}{\medskip \refstepcounter{subsubsection}
\noindent  {\bf Remark \thetheorem}.\rm}{\,}
\newenvironment{definition}{\medskip \refstepcounter{subsubsection}
\noindent  {\bf Definition \thetheorem}.\rm}{\,}

\newtheorem{theointro}{Theorem}

\def\Aut{\mathrm{Aut}}

\def\cO{\mathcal{O}}

\def\cH{\mathcal{H}}

\def\cB{\mathcal{B}}
\def\cL{\mathcal{L}}

\def\co{\mathcal{O}_0}

\def\del{\partial}
\def\delb{\overline\partial}

\def\h{\mathfrak{h}}

\def\g{\mathfrak{g}}

\def\R{\mathbb{R}}

\def\N{\mathbb{N}}

\def\P{\mathbb{P}}

\def\C{\mathbb{C}}

\def\uS{\underline{S}}

\def\om{\omega}
\def\ep{\varepsilon}

\def\>{\rangle}

\def\<{\langle}
\def\>{\rangle}
\def\a{\alpha}

\begin{document}

\title[Extremal metrics and lower bound of the modified K-energy]
{Extremal metrics and lower bound of the modified K-energy}
\author[Y. Sano]{Yuji Sano}
\author[C. Tipler]{Carl Tipler}
\address{2-39-1, Kurokami, Chuo-ku, Kumamoto, 860-8555, Japan, 
Graduate School of Science and Technology, Kumamoto University ; D\'epartement de math\'ematiques, 
Universit\'e du Qu\'ebec \`a Montr\'eal, 
Case postale 8888, succursale centre-ville,
Montr\'eal (Qu\'ebec), H3C 3P8}
\email{sano@sci.kumamoto-u.ac.jp ; carl.tipler@cirget.ca}

\date{\today}

\begin{abstract}
We provide a new proof of a result of X.X. Chen and G.Tian \cite{ct}: for a polarized extremal K\"ahler manifold, an extremal metric attains the minimum of the modified K-energy. The proof uses an idea of Chi Li \cite{li} adapted to the extremal metrics using some weighted balanced metrics.
%$\sigma$-balanced metrics introduced by the first author \cite{sa}.
\end{abstract}

\maketitle

\section{Introduction}
\label{secintro}
Extremal metrics were introduced by Calabi \cite{c1}. Let $(X,\om)$ be a K\"ahler manifold of complex dimension $n$. An extremal metric is a critical point of the functional
\begin{eqnarray*}
  g & \mapsto & \int_X (S(g))^2 \frac{\om_g^n}{n!}
\end{eqnarray*}
defined on K\"ahler metrics $g$ representing the K\"ahler class $[\om]$, where $S(g)$ is the scalar curvature of the metric $g$.
Constant scalar curvature K\"ahler metrics (CSCK for short), and in particular K\"ahler-Einstein metrics, are extremal metrics. In this work we will focus on the polarized case, assuming that there is an ample holomorphic line bundle $L\rightarrow X$ with $c_1(L)=[\om]$. In this special case, it has been conjectured by Yau in the K\"ahler-Einstein case \cite{yau},  and then in the CSCK case by the work of Tian \cite{tian} and Donaldson \cite{Don02} that the existence of a CSCK metric representing $c_1(L)$ should be equivalent to a GIT stability of the pair $(X,L)$. This conjecture has been extended to extremal metrics by Sz\'ekelyhidi \cite{sz} and Mabuchi \cite{ma12}.

\par
Let $(X,L)$ be a polarized K\"ahler manifold. Donaldson has shown \cite{Don01} that if $X$ admits a CSCK metric in $c_1(L)$, and if $\Aut(X,L)$ is discrete, then the CSCK metric can be approximated by a sequence of balanced metrics. This approximation result implies in particular the unicity of a CSCK metric in its K\"ahler class. 
This method has been adapted by Mabuchi \cite{ma04} to the extremal metric setting to prove unicity of an extremal metric up to automorphisms in a polarized K\"ahler class. Then, Chen and Tian proved unicity of an extremal metric in its K\"ahler class up to automorphisms with no polarization assumption \cite{ct}.

\par
In a sequel to his work on balanced metrics \cite{Don05}, Donaldson shows that if $\Aut(X,L)$ is discrete, a CSCK metric is an absolute minimum of the Mabuchi energy $E$, or K-energy, introduced by Mabuchi \cite{ma}.
The approximation result of Donaldson does not hold true for CSCK metrics if the automorphism group is not discrete. There are counter-examples of Ono, Yotsutani and the first author \cite{osy}, or Della Vedova and Zudas \cite{dz}. However, Li managed to show that even if $\Aut(X,L)$ is not discrete, a CSCK metric would provide an absolute minimum of $E$ \cite{li}.

\par
 By a theorem of Calabi \cite{c2}, extremal metrics are invariant under a maximal connected compact sub-group $G$ of the reduced automorphism group $\Aut_0(X)$ \cite{fuj}. Any two such compact groups are conjugated in $\Aut_0(X)$ and the study of extremal metrics is done modulo one such group.
In the extremal setting, the modified K-energy $E^G$ (see definition \ref{def:modK}) plays the role of the K-energy for CSCK metrics. This functional has been introduced independently by Guan \cite{gu}, Simanca \cite{si} and Chen and Tian \cite{ct} and is defined  on the space of $G$-invariant K\"ahler potentials with respect to a $G$-invariant metric. In \cite{ct}, Chen and Tian prove that extremal metrics minimize the modified K-energy up to automorphisms of the manifold, with no polarization assumption. In this paper, we give a different proof of this result in the polarized case. We generalize Li's work to extremal metrics, using some weighted balanced metrics, which are called $\sigma$-balanced metrics (see definition \ref{def:sigma-balanced} in section 2):
%building on the recent $\sigma$-balanced metrics introduced by the first author \cite{sa}:

\begin{theointro}
\label{theo:min}
Let $(X,L)$ be a polarized K\"ahler manifold and $G$ a maximal connected compact sub-group of the reduced automorphism group $\Aut_0(X)$. Then $G$-invariant extremal metrics representing $c_1(L)$ attain the minimum of the modified K-energy $E^G$.
\end{theointro}

The proof relies on two observations. We will consider a sequence of Fubini-Study metrics $\om_k$ associated to Kodaira embeddings of $X$ into higher and higher dimension projective spaces. The first observation is that if we define $\om_k$ to be the metric associated to an extremal metric in $c_1(L)$ by the map $Hilb_k$ (see definition in section~\ref{sec:background}, equation~(\ref{eq:Hilb})), then $\om_k$ will be close to a $\sigma$-balanced metric. The second point
%, observed by the first author in \cite{sa}, 
is that $\sigma$-balanced metrics, if they exist, are minima of the functionals $Z^{\sigma}_k$ (section~\ref{sec:background}, equation~(\ref{eq:Zk})) that converge to the modified Mabuchi functional. Then a careful analysis of the convergence properties of the $\om_k$ and $Z_k^{\sigma}$ yields the proof of our main result.

\begin{remark}
We shall mention that Guan shows in \cite{gu} that extremal metrics are local minima, assuming the existence of $C^2$-geodesics in the space of K\"ahler potentials.
\end{remark}

\subsection{Plan of the paper}
In section~\ref{sec:background}, we review the definition of extremal metrics and recall quantization of CSCK metrics. We then introduce $\sigma$-balanced metrics and the relative functionals. Then in section~\ref{sec:min}, we prove the main theorem. In the Appendix, we collect some facts and proofs of properties of $\sigma$-balanced metrics.

\subsection{Acknowledgments}
The first author is supported by MEXT, Grant-in-Aid for Young Scientists (B), No. 22740041.
The part of the article concerning $\sigma$-balanced metrics is presented by the first author in 2011 Complex Geometry and Symplectic Geometry Conference held in the University of Science and Technology in China.
He would like to thank the organizers for the invitation and the kind hospitality. 
The second author would like to thank Hongnian Huang, Vestislav Apostolov and Andrew Clarke for useful discussions and their interest in this work, as well as Song Sun and Valentino Tosatti for their  encouragement.

\section{Extremal metrics and Quantization}
\label{sec:background}
\subsection{Quantization}
                    
Let $(X,L)$ be a polarized K\"ahler manifold of complex dimension $n$. Let $\cH$ be the space of smooth K\"ahler potentials with respect to a fixed K\"ahler form $\om \in c_1(L)$  :
\begin{eqnarray*}
\cH= \lbrace \phi \in C^{\infty} (X) \;\vert \; \om_{\phi}:=\om + \sqrt{-1}\del\delb \phi > 0 \rbrace.
\end{eqnarray*}

For each $k$, we can consider $\cH_k$ the space of hermitian metrics on $L^{\otimes k}$. To each element $h\in \cH_k$ one associates a metric $ -\sqrt{-1} \del\delb \log(h)$ on $X$, identifying the spaces $\cH_k$ to $\cH$. Write $\om_h$ to be the curvature of the hermitian metric $h$ on $L$. Fixing a base metric $h_0$  in $\cH_1$ such that $\om=\om_{h_0}$ the correspondence reads
\begin{eqnarray*}
\om_{\phi}=\om_{e^{-\phi}h_0}=\om+\sqrt{-1}\del\delb \phi .
\end{eqnarray*}
We denote by $\cB_k$ the space of positive definite Hermitian forms on $H^0(X,L^{\otimes k})$. Let $N_k=dim(H^0(X,L^k))$.
The spaces $\cB_k$ are identified with $GL_{N_k}(\C)/ U(N_k)$ using the base metric $h_0^k$. These symmetric spaces come with metrics $d_k$ defined by  Riemannian metrics:
\begin{eqnarray*}
(H_1,H_2)_h=Tr(H_1H^{-1}\cdot H_2 H^{-1}).
\end{eqnarray*}

\noindent There are maps :
\begin{eqnarray*}
Hilb_k :  \cH & \rightarrow &\cB_k \\
FS_k : \cB_k &\rightarrow &\cH
\end{eqnarray*}
defined by :
\begin{eqnarray*}
\forall h\in \cH\;, \; s\in H^0(X,L^{\otimes k})\;, \; \vert\vert s\vert\vert ^2_{Hilb_k(h)}=\int_X \vert s \vert_{h^k}^2 d\mu_h
\end{eqnarray*}
and
\begin{eqnarray*}
\forall H \in \cB_k\; , \; 
FS_k(H)= \frac{1}{k} \log \bigg(\frac{1}{N_k}  \sum_{\alpha} \vert s_{\alpha}\vert_{h_0^k}^2\bigg)
\end{eqnarray*}
where  $\lbrace s_{\alpha}\rbrace$ is an orthonormal basis of $H^0(X,L^{\otimes k})$ with respect to $H$
and $d\mu_{h}=\dfrac{\om_{h}^n}{ n!}$ is the volume form. 
Note that $\om_{FS_k(H)}$ is the pull-back of the Fubini-Study metric on $\C\P_{N_k-1}$ under the projective embedding induced by $\lbrace s_{\alpha}\rbrace$.
A result of Tian \cite{tian90} states that any K\"ahler metric $\om_{\phi}$ in $c_1(L)$ can be approximated by projective metrics, namely
\begin{eqnarray*}
\lim_{k\rightarrow \infty} FS_k \circ Hilb_k (\phi) = \phi
\end{eqnarray*}
where the convergence is uniform on $C^2(X,\R)$ bounded subsets of $\cH$.

\noindent The metrics satisfying 
$$
FS_k\circ Hilb_k(\phi)=\phi
$$
are called balanced metrics, and the existence of such metrics is equivalent to the Chow stability of $(X,L^k)$ by Zhang \cite{zha} and Wang \cite{wa}.
Let $\Aut(X,L)$ be the group of automorphisms of the pair $(X,L)$.  
From the work of Donaldson \cite{Don01}, if $X$ admits a CSCK metric in the K\"ahler class $c_1(L)$, and if $\Aut(X,L)$ is discrete, then there are balanced metrics $\om_{\phi_k}$ for $k$ sufficiently large, with
$$
FS_k\circ Hilb_k(\phi_k)=\phi_k
$$ 
and these metrics converge to the CSCK metric on $C^{\infty}(X,\R)$ bounded subsets of $\cH$.

In the proof of these results, the density of state function plays a central role.
For any $\phi\in\cH$ and $k>0$, let $\lbrace s_{\alpha} \rbrace$ be an orthonormal basis of $H^0(X,L^k)$ with respect to $Hilb_k(\phi)$. The $k^{th}$ Bergman function of $\phi$ is defined to be :
$$
\rho_k(\phi)=\sum_{\alpha}\vert s_{\alpha}\vert^2_{h^k}.
$$
It is well known that a metric $\phi\in Hilb_k(\cH)$ is balanced if and only if $\rho_k(\phi)$ is constant. 
A key result in the study of balanced metrics is the following expansion:

\begin{theorem}[\cite{cat},\cite{ruan},\cite{tian90},\cite{zel}]
The following uniform expansion holds
$$
\rho_k(\phi)=k^n+A_1(\phi)k^{n-1}+A_2(\phi)k^{n-2}+...
$$
with $A_1(\phi)=\frac{1}{2}S(\phi)$ is half of the scalar curvature of the K\"ahler metric $\om_\phi$
and for any $l$ and $R\in \N$, there is a constant $C_{l,R}$ such that
$$
\vert\vert\rho_k(\phi) -\sum_{j\leq R} A_j k^{n-j} \vert \vert_{C^l} \leq k^{n-R}.
$$
\end{theorem}

\noindent As a corollary, if $\phi_k=FS_k\circ Hilb_k(\phi)$, then
$$
\phi_k-\phi=\frac{1}{k}\log \rho_k(\phi)\rightarrow 0
$$
as $k \rightarrow \infty$. 
In particular we have the convergence of metrics
\begin{equation}
\label{cor:ber}
\om_{\phi_k}=\om_{\phi}+O(k^{-2}).
\end{equation}
By integration over $X$ we also deduce
$$
\int_X \rho_k(\phi) d\mu_{\phi}=k^n\int_X d\mu_{\phi}+ k^{n-1}\frac{1}{2}\int_X S(\phi) d\mu_{\phi}+\cO(k^{n-2}) 
$$
where $S(\phi)$ is the scalar curvature of the metric $g_{\phi}$ associated to the K\"ahler form $\om_{\phi}$. 
Thus
\begin{equation}
\label{cor:Nk}
N_k=k^n\mathrm{Vol}(X)+\frac{1}{2} \mathrm{Vol}(X) \uS k^{n-1}+\cO(k^{n-2}).
\end{equation}
where
$$
\uS=2n\pi\frac{c_1(L)\cup[\om]^{n-1}}{[\om]^n}
$$
is the average of the scalar curvature and $\mathrm{Vol}(X)$ is the volume of $(X,c_1(L))$.
\subsection{The relative setup}
In order to find a canonical representative of a K\"ahler class, Calabi suggested \cite{c1} to look for minima of the functional
\begin{eqnarray*}
 Ca :  \cH & \to & \R \\
  \phi & \mapsto & \int_X (S(\phi)-\uS)^2 d\mu_{\phi}.
\end{eqnarray*}
In fact, critical points for this functional are local minima, called extremal metrics. 
The associated Euler-Lagrange equation is equivalent to the fact that $grad_{\om_{\phi}}(S(\phi))$ is a holomorphic vector field and constant scalar curvature metrics, CSCK for short, are extremal metrics.
\par
By a theorem of Calabi \cite{c2}, the connected component of identity of the isometry group of an extremal metric is a maximal compact connected subgroup of $\Aut_0(X)$.  As all these maximal subgroups are conjugated, the quest for extremal metrics can be done modulo a fixed group action. Note that $\Aut_0(X)$ is isomorphic to $\Aut_0(X,L)$ the connected component of identity of $\Aut(X,L)$.
As we will see later, it will be useful to consider a less restrictive setup, working modulo a circle action. We then define the relevant functionals in a general situation and we fix $G$ a compact subgroup of $\Aut_0(X,L)$ and denote by $\g$ its Lie algebra.

\subsubsection{Space of potentials}
We extend the quantization tools to the extremal metrics setup. 
\par
 Replacing $L$ by a sufficiently large tensor power if necessary, we can assume that $\Aut_0(X,L)$ acts on $L$ (see e.g. \cite{kob}).
Then the $G$-action on $X$ induces a $G$-action on the space of sections $H^0(X,L^k)$. This action in turn provides a $G$-action on the space  $\cB_k$  of positive definite hermitian forms on $H^0(X,L^k)$ and we define $\cB_k^G$ to be the subspace of $G$-invariant elements.
The spaces $\cB_k^G$ are totally geodesic in $\cB_k$ for the distances $d_k$.  
Define $\cH^G$ to be the space of $G$-invariant potentials with respect to a $G$-invariant base point $\om$. 
We see from their definitions that we have the induced maps :
\begin{equation}
\label{eq:Hilb}
\begin{array}{cccc}
Hilb_k : & \cH^G & \rightarrow &\cB_k^G \\
FS_k :& \cB_k^G &\rightarrow &\cH^G.
\end{array}
\end{equation}

\subsubsection{Modified K-energy}
For a fixed metric $g$, we say that a vector field $V$ is a hamiltonian vector field if there is a real valued function $f$ such that
$$
V=J\nabla_g f
$$
or equivalently
$$
\om(V,\cdot)= -df.
$$
For any $\phi\in\cH^G$, let $P_{\phi}^G$ be the space of normalized (i.e. mean value zero) Killing potentials with respect to $g_{\phi}$ whose corresponding hamiltonian vector fields lie in $\g$ and let $\Pi_{\phi}^G$ be the orthogonal projection from $L^2(X,\R)$ to $P_{\phi}^G$ given by the inner product on functions 
$$
(f,g) \mapsto \int fg d\mu_{\phi}.
$$
\noindent Note that $G$-invariant metrics satisfying $S(\phi)-\uS-\Pi_{\phi}^G S(\phi)=0$ are extremal.

\begin{definition}\cite[Section 4.13]{gbook}
The reduced scalar curvature $S^G$ with respect to $G$ is defined by
$$
S^G(\phi)=S(\phi)-\uS-\Pi_{\phi}^G S(\phi).
$$
The extremal vector field $V^G$ with respect to $G$ is defined by the equation
$$
V^G=\nabla_g (\Pi_\phi^G S(\phi))
$$
for any $\phi$ in $\cH^G$ and does not depend on $\phi$ (see e.g. \cite[Proposition 4.13.1]{gbook}).
\end{definition}

\begin{remark}
Note that by definition the extremal vector field is real-holomorphic and lies in $J\g$ where $J$ is the almost-complex structure of $X$, while $JV^G$ lies in $\g$.
\end{remark}

\begin{remark}
When $G=\lbrace 1 \rbrace$ we recover the normalized scalar curvature. When $G$ is a maximal compact connected subgroup, or maximal torus of $\Aut_0(X)$, we find the reduced scalar curvature
and the usual extremal vector field initially defined by Futaki and Mabuchi \cite{fm}.
\end{remark}

\par
We are now able to define the relative Mabuchi K-energy, introduced by Guan \cite{gu}, Chen and Tian \cite{ct}, and Simanca \cite{si}:

\begin{definition}\cite[Section 4.13]{gbook} 
\label{def:modK}
The modified Mabuchi K-energy $E^G$ (relative to $G$) is defined, up to a constant, as the primitive of the following one-form on $\cH^G$:
$$
\phi \mapsto - S^G(\phi) d\mu_{\phi}.
$$
\end{definition}

\noindent If $\phi\in \cH^G$, then the modified K-energy admits the following expression
$$
E^G(\phi)=-\int_X \phi (\int_0^1S^G(t\phi) d\mu_{t\phi} dt) .
$$
As for CSCK metrics, $G$-invariant extremal metrics whose extremal vector field lie in $J\g$ are critical points of the relative Mabuchi energy.

\subsubsection{The $\sigma$-balanced metrics}

We present a generalization of balanced metrics adapted to the relative setting of extremal metrics.

\begin{definition}\label{def:sigma-balanced}
Let $\sigma_k(t)$ be a one-parameter subgroup of $\Aut_0(X,L^k)$.
Let $\phi \in \cH$.  Then $\phi$ is a $k^{th}$ $\sigma_k$-balanced metric if
\begin{equation}
\label{def:bal}
\om_{kFS_k\circ Hilb_k (\phi)}=\sigma_k(1)^*\om_{k\phi}
\end{equation}
\end{definition}

Conjecturally, the $\sigma$-balanced metrics would provide the generalization of the notion of balanced metric and would approximate an extremal K\"ahler metric. Indeed, in one direction, assume that we are given $\sigma_k$-balanced metrics $\om_{\phi_k}$, with $\sigma_k\in \Aut_0(X,L^k)$ such that the $\om_k$ converge to $\om_{\infty}$.
Suppose that the vector fields $k\frac{d}{dt}\vert_{t=0}\sigma_k(t)$ converge to a vector field $V_{\infty}\in \h_0$. 
A simple calculation implies that $\om_{\infty}$ must be extremal.
% Then it is shown in \cite{sa} that $\om_{\infty}$ must be extremal.

\par

We now define the functionals that play the role of finite dimensional versions
of the modified Mabuchi K-energy on $\cB_k^G$ and $FS_k(\cB_k^G)$ respectively.
First define $I_k=\log\circ \det$ on $\cB_k^G$. This functional is defined up to an additive constant when we see $\cB_k^G$ as a space of positive Hermitian matrix
once a suitable basis of $H^0(X,L^k)$ is fixed. 
It is shown in \cite{cs} that $I_k$ gives a quantization of the Aubin functional $I$.
However in the extremal case, we need a modified version of the Aubin functional defined by the first author in order to fit with the balanced metrics. Let $V\in Lie(\Aut_0(X,L))$ and denote by $\sigma(t)$ the associated one parameter subgroup of $\Aut_0(X,L)$. Define
up to a constant for each $\phi\in \cH$ the function $\psi_{\sigma,\phi}$ by
\begin{equation}
\label{def:psi}
\sigma(1)^*\om_{\phi}=\om_{\phi}+\sqrt{-1} \del\delb \psi_{\sigma,\phi}.
\end{equation}
We will see in the sequel how to choose suitably a normalization constant for these potentials.
We then consider a modified $I$ functional defined up to a constant by its
differential:
$$
\delta I^{\sigma}(\phi)(\delta\phi)=\int_X \delta \phi (1+\Delta_{\phi})e^{\psi_{\sigma,\phi}} d\mu_{\phi}
$$
where $\Delta_{\phi}=-g_{\phi}^{i\overline{j}}\frac{\del}{\del z_i}\frac{\del}{\del \overline{z}_j}$ is the complex Laplacian of $g_{\phi}$. We will also need to consider the potentials $\phi$ as metrics on the tensor powers $L^{\otimes k}$, we thus consider the normalized vector fields $V_k=-\frac{V}{4k}$ and the associated one-parameter groups $\sigma_k(t)$. 
We choose the normalization

\begin{equation}
\label{eq:norm}
\int_X \exp{(\psi_{\sigma_k,\phi})}\; d\mu_{\phi} = \frac{N_k}{k^n}
\end{equation}
Then we define for each $k$
$$
\delta I^{\sigma}_k(\phi)(\delta\phi)=\int_X k\delta \phi (1+\frac{\Delta_{\phi}}{k})e^{\psi_{\sigma_k,\phi}} k^n d\mu_{\phi}.
$$

\begin{remark}
If $\sigma$ is the identity, we recover the usual Aubin functional.
\end{remark}

\begin{remark}
This one-form integrates along paths in $\cH^G$ to a functional $I_k^{\sigma}(\phi)$
on
$\cH^G$, which is independent on the path used from $0$ to $\phi$.
The proof of this fact is given in the Appendix, proposition \ref{prop:indep}.
\end{remark}

%Then, following \cite{sa}, 
We define
$\cL_k^{\sigma}$ on $\cH^G$ and $Z^{\sigma}_k$ on $\cB_k^G$ by
\begin{equation}
\label{eq:Lk}
\cL^{\sigma}_k = I_k\circ Hilb_k + I_k^{\sigma}
\end{equation}
and
\begin{equation}
\label{eq:Zk}
Z^{\sigma}_k =  I_k^{\sigma}\circ FS_k + I_k-k^n\log(k^n)\mathrm{Vol}(X).
\end{equation}
We will show in the following that these functionals converge to the modified $K$-energy in some sense. Note also that $\sigma_k$-balanced metrics are critical points for $\cL_k^{\sigma}$ (proposition \ref{prop:criticL}) and, if $FS_k(H_k)$ is a $\sigma_k$-balanced metric for some $H_k\in \cB_k^G$, then $H_k$ is a minimum for $Z_k^{\sigma}$ (proposition \ref{prop:Z_min}).

\section{Minima of the modified K-energy}
\label{sec:min}
The aim of this section is to prove Theorem \ref{theo:min}. For the convenience of the reader we give a sketch of the proof.

We will choose the special group $G$ corresponding to the Killing field $JV^*$ associated to the extremal vector field $V^*$ of the extremal K\"ahler metric $\om^*=\om_{\phi^*}$. We know that the metrics $\om_k^*=\om+\sqrt{-1}\del\delb \phi_k^*$ with K\"ahler potentials 
$\phi^*_k=FS_k\circ Hilb_k (\phi^*)$ converge to $\om^*$ (\cite{tian90}, \cite{cat} and \cite{zel}). 
We begin our proof by showing
%recalling  the proof of the first author 
that the functionals $\cL_k^{\sigma}$ converge to the modified Mabuchi functional on the space $\cH^G$. Then we show that  $Z_k^{\sigma}\circ Hilb_k$ and $\cL_k^{\sigma}$ converge to the  same functional, thus $Z_k^{\sigma}$ gives a quantization of the modified Mabuchi functional and we reduce our problem to studying the minima of $Z_k^{\sigma}$.
However the metrics $\om_k^*$ constructed above are not in general critical points of  $Z_k^{\sigma}$, as there is no reason for these metrics to be $\sigma_k$-balanced. We use instead an idea of Li \cite{li} relying on the Bergman kernel expansion to show that these metrics $\om_k^*$ are almost $\sigma_k$-balanced metrics, in the sense that $Hilb_k(\om_k^*)$ is a minimum of the functional $Z_k^{\sigma}$ up to an error which goes to zero when $k$ tends to infinity. 

\par

Let $V^*$ be the extremal vector field of the class $c_1(L)$. In the polarized case, the vector field $JV^*$ generates a periodic action \cite{fm} by a one parameter-subgroup of automorphisms of $(X,L)$. Fix $G$ to be the one-parameter subgroup of $\Aut(X,L)$ associated to $JV^*$.
This group is isomorphic to $S^1$ or trivial by the theorem of Futaki and Mabuchi \cite{fm}. This will be a group of isometries for each of our metrics. 

\begin{remark}
\label{rmk:mab}
The modified K-energy $E^{G_m}$ is defined to be the modified Mabuchi functional with respect to a maximal compact connected subgroup $G_m$ of $\Aut(X,L)$. Assume that $G$ is contained in such a $G_m$. Then $E^{G_m}$ is equal to $E^G$ when restricted to the space of $G_m$-invariant potentials.
Indeed, the projection of any $G_m$-invariant scalar curvature to the space of holomorphy potentials of $Lie(G_m)$ gives a potential for the extremal vector field by definition.
Thus a minimum of $E^G$ which is invariant under the $G_m$-action, such as an extremal metric, will be a minimum of the usual modified Mabuchi functional. 
\end{remark}

Let  $\sigma_k$ be the element of $\Aut(X,L)$ associated to the vector field $-\frac{V^*}{4k}$. We will also need to define for each $\phi$
in $\cH^G$ the function
$
\theta(\phi)
$
to be the normalized (i.e. mean value zero) holomorphy potential of the vector field $V^*$ with respect to the metric $\om_{\phi}$:
$$
g_\phi(V^*,\cdot )=d \theta(\phi)
$$
or
$$
\theta (\phi)=\Pi_\phi^G(S(\phi)).
$$

\subsection{The functionals $\cL_k^{\sigma}$ converge to $E^G$}
In this section we prove the following fact :

\begin{proposition}
\label{prop:LquantizeE}
There are constants $c_k$ such that 
$$
\frac{2}{k^n}\cL_k^{\sigma} + c_k \rightarrow E^G
$$
as $k\rightarrow \infty$, where the convergence is uniform on $C^l(X,\R)$ bounded subsets of $\cH^G$.
\end{proposition}

\begin{proof}
We show that
$$
\frac{2}{k^n}\delta \cL^{\sigma}_k  \rightarrow \delta E^G
$$
uniformly on $C^l(X,\R)$ bounded subsets of $\cH^G$. First we compute $\delta \cL^{\sigma}_k$. Following \cite{Don05}:
$$
\delta (I_k \circ Hilb_k)_{\phi}(\delta \phi)=- \int_X \delta \phi (\Delta_{\phi}+k) \rho_k(\phi) d\mu_{\phi}
$$
and by definition
$$
\delta (I_k ^{\sigma})_{\phi}(\delta \phi)=k^n \int_X \delta \phi (k+\Delta_{\phi})e^{\psi_k(\phi) } d\mu_{\phi}
$$
where we set $\psi_k(\cdot)=\psi_{\sigma_k,\cdot}$.
\\
Then
\begin{equation}
\label{eq:diffL}
\delta (\cL^{\sigma}_k)_{\phi}(\delta \phi)=-\int_X \delta \phi (\Delta_{\phi}+k)(\rho_k(\phi)-k^n e^{\psi_k(\phi) }) d\mu_{\phi}.
\end{equation}
We need an expansion for the potential $\psi_k$ :
$$
\psi_k(\phi)=\frac{\theta(\phi)+\uS}{2k}+\co(k^{-1})
$$
which proof is postponed to lemma \ref{lem:exppsi}.
Then by the expansions of $\psi_k(\phi)$ and $\rho_k(\phi)$
$$
(\Delta_{\phi}+k)(\rho_k(\phi)-k^n e^{\psi_k(\phi) })=k^n(\Delta_{\phi}+k)(1+\frac{S(\phi)}{2k}+\cO(k^{-2})-1-\frac{\theta(\phi)+\uS}{2k}+\co\mathcal{}(k^{-1})),
$$
$$
(\Delta_{\phi}+k)(\rho_k(\phi)-k^n e^{\psi_k(\phi) })=k^n (\frac{S(\phi)-\uS-\theta(\phi))}{2} +\cO(k^{-1})),
$$
and
$$
\frac{\delta (\cL^{\sigma}_k)_{\phi}}{k^n} \rightarrow \frac{1}{2}\delta E^G_{\phi} .
$$
As the expansions of $\psi_k(\phi)$ and $\rho_k(\phi)$ are uniform on bounded subsets of $C^l(X,\R)$ the result follows.
\end{proof}

\noindent The following lemma will be useful  :

\begin{lemma}
\label{lem:exppsi}
The following expansion holds uniformly in $C^l(X,\R)$ for $l>>1$:
\begin{equation}
\label{eq:exppsi}
\psi_k(\phi)=\frac{\theta(\phi)+\uS}{2k}+\co(k^{-1})
\end{equation}
where $\co(k^{-1})$ denotes $k^{-1}$-times a function $\ep(k)$ on $X$ with $\ep(k)\rightarrow 0$ in  $C^l(X,\R)$ as $k \rightarrow 0$.
\end{lemma}

\begin{proof}
By definition
$$
\sigma_k(1)^*\om(\phi)-\om(\phi)=\sqrt{-1} \del \delb \psi_k(\phi),
$$
then
$$
\sigma_1(\frac{1}{k})^*\om(\phi)-\om(\phi)=\sqrt{-1} \del \delb \psi_k(\phi),
$$
where $\sigma_1(\frac{1}{k})$ is equal to $\exp(-\frac{1}{4k}V^*)$.
Dividing by $\dfrac{1}{k}$, and letting $k$ go to infinity,
$$
\cL_{-\frac{1}{4}V^*}\om(\phi)=\sqrt{-1}\del\delb \lim_{k\rightarrow \infty}(k\psi_k(\phi))
$$
Then by Cartan's formula,
$$
\begin{array}{ccc}
\cL_{-\frac{1}{4}V^*}\om(\phi) & = & -\frac{1}{4}d\om_{\phi}(V^*,\cdot) \\
 & = & -\frac{1}{4}d g_{\phi}(V^*,J \cdot) \\
 \end{array}
$$
and by definition of holomorphy potentials 
$$
\cL_{-\frac{1}{4}V^*}\om(\phi)=-\frac{1}{4}d d^c \theta(\phi)=\frac{\sqrt{-1}}{2}\del\delb \theta(\phi)
$$
thus
$$
\lim_{k\rightarrow \infty}(k\psi_k(\phi))=\frac{\theta(\phi)+c}{2}
$$
for some constant $c$.
By the normalization~(\ref{eq:norm}) of the function $\psi_k(\phi)$ we deduce
$$
\frac{N_k}{k^n}=\int_X \exp{(\psi_{\sigma_k,\phi})}\; d\mu_{\phi} =  \int_X 1+\frac{\theta(\phi)+c}{2k}+\cO(k^{-2}) d\mu_{\phi}.
$$
Recall that we choose $\theta(\phi)$ normalized to have mean value zero. Using formula (\ref{cor:Nk}) to expand $N_k=dim(H^0(X,L^k))$, we conclude  that $c=\uS$.
\end{proof}
\noindent From the above computations we also deduce the following :
\begin{proposition}
\label{prop:criticL}
Let $\phi \in \cH$ be a $k^{th}$ $\sigma_k$-balanced metric. Then $\phi$ is a critical point of $\cL_k^{\sigma}$.
\end{proposition}

\begin{proof}
By equation $(\ref{def:bal})$ of $\sigma_k$-balanced metrics and by definition (\ref{def:psi}) of $\psi_k(\phi)$ we deduce
$$
\rho_k(\phi)=C\exp(\psi_k(\phi))
$$
for some constant $C$. Integrating over $X$ and using the expansions $(\ref{cor:Nk})$ and $(\ref{eq:exppsi})$ we deduce
$$
\rho_k(\phi)=k^n\exp(\psi_k(\phi)).
$$
The result follows from the computation of the differential of $\cL^{\sigma}_k$, equation $(\ref{eq:diffL})$.
\end{proof}
\noindent
A direct computation implies the similar result for $Z^\sigma_k$ (see proposition \ref{prop:Z_min} in the Appendix).

\subsection{Comparison of $Z_k^{\sigma}$ and $\cL_k^{\sigma}$}

The aim of this section is to show that $Z_k^{\sigma}\circ Hilb_k$ and $\cL_k^{\sigma}$ converge to the same functional. We will need the two following lemmas:

\begin{lemma}
\label{lem:hessian}
The second derivative of $I_k^{\sigma}$ along a path $\phi_s\in\cH^G$
is equal to
$$
\frac{d^2}{ds^2}I_k^{\sigma}(\phi_s)=k^n\int_X (\phi''-\frac{1}{2} \vert d\phi'\vert^2)
(k+\Delta_{\phi_s})e^{\psi_k(\phi_s)}d\mu_{\phi_s}
$$
\end{lemma}

\begin{proof}
The proof of this result is given in the Appendix, section \ref{sec:hessian}.
\end{proof}

\begin{lemma}
\label{lem:concave}
Let $\phi\in\cH^G$. Then there exists an integer $k_0$, depending on $\phi$,
such that for each $k\geq k_0$, the functional $I_k^{\sigma}$ is concave along
the path
$$
\begin{array}{ccc}
[0,1] & \rightarrow & \cH^G \\
s & \mapsto & \phi+\frac{s}{k}\log(\rho_k(\phi))
\end{array}
$$
\end{lemma}

\begin{proof}
By lemma~\ref{lem:hessian}, the second derivative of $I_k^{\sigma}$ along
the path $\phi_k(s)=\phi+\frac{s}{k}\log(\rho_k(\phi))$ is
$$
k^n\int_X (\phi_k''-\frac{1}{2} \vert d\phi_k'\vert^2)
(k+\Delta_{\phi_k(s)})e^{\psi_k(\phi_k(s))}d\mu_{\phi_k(s)}.
$$
As $\phi_k'=\frac{1}{k}\log(\rho_k(\phi))$ and $\phi_k''=0$, this expression simplifies:
$$
\frac{d^2}{ds^2} I_k^{\sigma} (\phi_k(s))=-k^n\int_X \frac{1}{2} \vert d \frac{1}{k}\log(\rho_k(\phi)) \vert^2 (k+\Delta_{\phi_k(s)})e^{\psi_k(\phi_k(s))}d\mu_{\phi_k(s)}.
$$
We compute the leading term in the above expression as $k$ goes to
infinity. To simplify notation, let $T_k(\phi)=FS_k\circ Hilb_k(\phi)$. Note that $\om_{\phi_1}=\om_{T_k(\phi)}$. From (\ref{cor:ber}), the difference between $\om_{\phi_0}$
and $\om_{\phi_1}$ is $$\om_{\phi_0}-\om_{\phi_1}=\cO(k^{-2}).$$ 
Thus we have the estimates
$$
\Delta_{\phi_k(s)}=\Delta_{\phi}+\cO(k^{-1}),
$$
$$
d\mu_{\phi_k(s)}=d\mu_{\phi}+\cO(k^{-1})
$$
and
$$
\psi_k(\phi_k(s))=\psi_k(\phi)+\cO(k^{-1}).
$$
Then$$
\frac{d^2}{ds^2} I_k^{\sigma} (\phi_k(s))=-k^n\int_X \frac{1}{2} \vert d \frac{1}{k}\log(\rho_k(\phi)) \vert^2 (k+\Delta_{\phi})e^{\psi_k(\phi)} d\mu_{\phi} + \cO(k^{n-4}).
$$
From this we deduce that the leading term as $k$ tends to infinity is
$$
-\frac{k^{n-3}}{8} \int_X \vert dS(\phi)\vert ^2 d\mu_{\phi}<0
$$
where once again we used the expansions of Bergman kernel and of $\psi_k(\phi)$ from lemma \ref{lem:exppsi}.
\end{proof}

\noindent Now we can prove the main result of this section:

\begin{proposition}
\label{prop:compare}
For each potential $\phi\in\cH^G$, we have
$$
\lim_{k\rightarrow \infty} k^{-n}(\cL_k^{\sigma}(\phi)-Z_k^{\sigma}\circ Hilb_k (\phi)) =0 
$$
\end{proposition}

\begin{proof}
By definition,
$$
k^{-n}(\cL_k^{\sigma}(\phi)-Z_k^{\sigma}\circ Hilb_k (\phi)) = 
-k^{-n}( I_k^{\sigma} (T_k(\phi)) - I_k^{\sigma}(\phi)   -k^n\log(k^n)\mathrm{Vol}(X))
$$
where $T_k=FS_k\circ Hilb_k$.
From lemma~\ref{lem:concave}, for $k$ large enough, the functional $I_k^{\sigma}$ is concave along the path
$$
s  \mapsto  \phi+\frac{s}{k}\log(\rho_k(\phi))
$$
going from $\phi$ to $T_k(\phi)$ in $\cH^G$.

\noindent Thus
\begin{equation}
\label{eq:concave}
(\delta I_k^{\sigma})_{\phi}(\frac{1}{k}\log \rho_k(\phi)) 
 \geq (I_k^{\sigma} (T_k(\phi)) - I_k^{\sigma}(\phi))
  \geq (\delta I_k^{\sigma})_{T_k(\phi)}(\frac{1}{k}\log \rho_k(\phi)) .
\end{equation}

\noindent We deduce from the definitions that
\begin{equation}
\label{eq:concave}
k^{-n}(\cL^{\sigma}_k(\phi)-Z^{\sigma}_k\circ Hilb_k (\phi))
 \geq 
 -k^{-n}(\delta I_k^{\sigma})_{\phi}(\frac{1}{k}\log \rho_k(\phi)) + \log(k^n)\mathrm{Vol}(X)
\end{equation}
and
\begin{equation}
\label{eq:concave2}
 -k^{-n}(\delta I_k^{\sigma})_{T_k(\phi)}(\frac{1}{k}\log \rho_k(\phi)) + \log(k^n)\mathrm{Vol}(X)
 \geq 
k^{-n}(\cL^{\sigma}_k(\phi)-Z^{\sigma}_k\circ Hilb_k (\phi))
\end{equation}
and it remains to show that the right hand side of (\ref{eq:concave}) and the left hand side of (\ref{eq:concave2}) tend to zero.
First
$$
k^{-n}(\delta I_k^{\sigma})_{\phi}(\frac{1}{k}\log \rho_k(\phi)) - \log(k^n)\mathrm{Vol}(X) =
\int_X (\frac{1}{k}\log(\rho_k(\phi)))(k+\Delta_{\phi})e^{\psi_k(\phi)} d\mu_{\phi} -\mathrm{Vol}(X)\log(k^n)
$$
$$
=\int_X (\log(k^n)+\frac{S(\phi)}{2k}+\cO(k^{-2}))(1+\frac{\Delta_{\phi}}{k})(1+\frac{\theta(\phi)+\uS}{2k}+\co(k^{-1})) d\mu_{\phi} -\mathrm{Vol}(X)\log(k^n)
$$
by the expansion of Bergman kernel and lemma \ref{lem:exppsi}.
If follows that 
$$
k^{-n}(\delta I_k^{\sigma})_{\phi}(\frac{1}{k}\log \rho_k(\phi)) - \log(k^n)\mathrm{Vol}(X) =\mathrm{Vol}(X)\log(k^n) + \cO(k^{-1}) - \mathrm{Vol}(X)\log(k^n) \rightarrow 0
$$
as $k\rightarrow \infty$.
\par
Note that we didn't make use of the fact that the derivative $\delta I_k^{\sigma}$
was evaluated at $\phi$, so the above argument extends to the last term of the inequality (\ref{eq:concave2}), evaluated at $T_k(\phi)$, which thus tends to zero as well.
This ends the proof.
\end{proof}

\subsection{The metrics $Hilb_k(\om^*)$ are almost $\sigma$-balanced }

We will need the following convexity property of $Z_k^{\sigma}$:

\begin{lemma}
\label{lem:Zconvex}
The functional $Z^{\sigma}_k$ is convex along geodesics in $\cB_k^G$.
\end{lemma}

\begin{proof}
We follow the proof of Proposition 1 in \cite{Don05} (also Lemma 3.1 in \cite{PS03}).
Here we abbreviate the subscript $k$.
Take a geodesic $\{H(s)\}_{s\in \mathbb{R}}$ in $\mathcal{B}^G$.
By choosing an appropriate orthonormal basis $\{\tau_\alpha\}$ of $H(0)$, $H(s)$ is represented by
\begin{equation}\label{eq:diagonalize}
	H(s)=diag(e^{2\lambda_\alpha s}), \,\,\, \lambda_\alpha\in \mathbb{R},
	\,\, \sum_\alpha \lambda_\alpha=0
\end{equation}
with respect to the basis $\{ \tau_\alpha\}$.
We denote the associated one parameter subgroup of $SL(H^0(M,L))$ by $\varrho(s)$.
We denote the K\"ahler potential $\phi_s=FS(H(s))$ by 
$$
	\phi_s=\log \big(\sum_\alpha |\varrho(s) \cdot\tau_\alpha|^2/\sum_\beta |\tau_\beta|^2\big).
$$

First of all, we will show the first variation of $Z^\sigma$ along $\phi_s$.
From (\ref{eq:derivative_psi}), we have
\begin{eqnarray}
	\nonumber
		\frac{d Z^\sigma}{ds} (s)
	&=&
		\int_X
		\phi'_{s}
		(1+\Delta_{FS(H(s))})
		e^{\psi_s}
		d\mu_{FS(H(s))}
	\\
	\nonumber
	&=&
		\int_X
		\phi'_{s}e^{\psi_s}
		+
		\frac{d}{ds} e^{\psi_s}
		d\mu_{FS(H(s))}
	\\
	\nonumber
	&=&
		\int_X
		\frac{\sum_\alpha
		2
		\lambda_\alpha |\varrho(s)\cdot\tau_\alpha|^2}{\sum_\beta |\varrho(s)\cdot\tau_\beta|^2}
		\bigg(\frac{\sum_\gamma |\varrho(s)\cdot\sigma^* \tau_\gamma|^2}{\sum_\beta |\tau_\beta|^2}\bigg)
	\\
	\nonumber
	&& \qquad\qquad
		+
		\bigg\{\frac{d}{ds}
		\bigg(\frac{\sum_\gamma |\varrho(s)\cdot\sigma^* \tau_\gamma|^2}{\sum_\beta |\varrho(s)\cdot \tau_\beta|^2}\bigg)
		\bigg\}
		d\mu_{FS(H(s))}
	\\
	\label{eq:1st_variation_Z}
	&=&
		\int_X
		\frac{\sum_\alpha 
		2
		\widetilde{\lambda}_\alpha|\varrho(s)\cdot\sigma^* \tau_\alpha|^2}{\sum_\beta |\varrho(s)\cdot\tau_\beta|^2}
		d\mu_{FS(H(s))},
\end{eqnarray}
where $\psi_s$ denotes $\psi_{\sigma, FS(H(s))}$.
In (\ref{eq:1st_variation_Z}), $H(s)$ is represented by
$$
	H(s)=diag(e^{2\widetilde{\lambda}_\alpha s}), \,\,\, \lambda_\alpha\in \mathbb{R},
	\,\, \sum_\alpha \widetilde{\lambda}_\alpha =0
$$
with respect to the basis $\{ \sigma^*\tau_\alpha\}$.
Let
$$
	\varphi'_s:=\frac{\sum_\alpha 2 \widetilde{\lambda}_\alpha |\varrho(s)\cdot(\sigma^* \tau_\alpha)|^2}{\sum_\beta |\varrho(s)\cdot\tau_\beta|^2}.
$$
Then, we have
\begin{equation}\label{eq:2nd_vari_Z(2)}
	\frac{d^2 Z^{\sigma}}{ds^2}(0)
	=
	\int_{X} \big\{\varphi''_0 -(\nabla \varphi'_0, \nabla \phi'_0)\big\} d\mu_{FS(H(0))}.
\end{equation}
Here we denote the connection of type $(1,0)$ by $\nabla$.
Following \cite{Don05}, it is sufficient to prove that the integrand of (\ref{eq:2nd_vari_Z(2)}) is equal to
\begin{equation}\label{eq:integrand}
	\sum_\alpha|(\nabla \phi'_0, \nabla (\sigma^* \tau_\alpha)) - (2\widetilde{\lambda}_\alpha- \phi'_0)(\sigma^*\tau_\alpha) |_{FS(H(0))}^2
\end{equation}
pointwise on $X$.
Expanding out, (\ref{eq:integrand}) is equal to
\begin{eqnarray}
	\nonumber
	&&
		\sum_\alpha |(\nabla \phi'_0, \nabla (\sigma^* \tau_\alpha))|_{FS(H(0))}^2
		-2 
		\sum_\alpha(2\widetilde{\lambda}_\alpha - \phi'_0)((\nabla \phi'_0, \nabla(\sigma^* \tau_\alpha)), \sigma^*\tau_\alpha)
	\\
	\label{eq:expanding}
	&& \qquad
		+ \sum_\alpha(2\widetilde{\lambda}_\alpha - \phi'_0)^2|\sigma^* \tau_\alpha|_{FS(H(0))}^2.
\end{eqnarray}
The second term of (\ref{eq:expanding}) is equal to
\begin{eqnarray}
	\nonumber
		-2
		\sum_\alpha(2\widetilde{\lambda}_\alpha - \phi'_0)(\nabla \phi'_0, ( \sigma^*\tau_\alpha, \nabla(\sigma^* \tau_\alpha)))
	&=&
		-2 
		\sum_\alpha(2\widetilde{\lambda}_\alpha - \phi'_0)
		(\nabla \phi'_0,\nabla(|\sigma^*\tau_\alpha|_{FS(H(0))}^2))
	\\
	\nonumber
	&=&
		- 2(\nabla \phi'_0, \nabla \varphi'_0)
		+
		2\phi'_0(\nabla \phi'_0, \nabla e^{\psi_0})
	\\
	\nonumber
	&=&
		- 2(\nabla \phi'_0, \nabla \varphi'_0)
		+
		2\phi'_0 \psi'_0 e^{\psi_0}
	\\
	\label{eq:2ndterm}
	&=&
		- 2(\nabla \phi'_0, \nabla \varphi'_0)
		+
		2\phi'_0 (\varphi'_0 - \phi'_0 e^{\psi_0}).
\end{eqnarray}
In above, we use (\ref{eq:derivative_psi}) in the Appendix and
$$
		\psi'_0 e^{\psi_0}
	=
		\frac{d}{ds}\bigg|_{s=0} e^{\psi_s}=\varphi'_0 - \phi'_0 e^{\psi_0}.
$$
The third term of (\ref{eq:expanding}) is equal to
\begin{equation}\label{eq:3rdterm}
	\sum_\alpha 4 \widetilde{\lambda}_\alpha^2 |\sigma^*\tau_\alpha|_{FS(H(0))}^2
	- 2 \varphi'_0 \phi'_0 + (\phi'_0)^2 e^{\psi_0}.
\end{equation}
Substituting (\ref{eq:2ndterm}) and (\ref{eq:3rdterm}) into (\ref{eq:expanding}), (\ref{eq:expanding}) is equal to
$$
		\sum_\alpha |(\nabla \phi'_0, \nabla (\sigma^* \tau_\alpha))|_{FS(H(0))}^2
		- 2(\nabla \phi'_0, \nabla \varphi'_0)
		-
		(\phi'_0)^2 e^{\psi_0}
		+
		\sum_\alpha 4 \widetilde{\lambda}_\alpha^2 |\sigma^*\tau_\alpha|_{FS(H(0))}^2.
$$
Since
$$
	\varphi''_0
	=
	\sum_\alpha 4 \widetilde{\lambda}_\alpha^2 |\sigma^*\tau_\alpha|_{FS(H(0))}^2
	- \varphi'_0 \phi'_0,
$$
the remain to be proved is 
\begin{equation}\label{eq:remain}
	\sum_\alpha |(\nabla \phi'_0, \nabla (\sigma^* \tau_\alpha))|_{FS(H(0))}^2
	=
	(\nabla \phi'_0, \nabla \varphi'_0)
	+(\phi'_0)^2 e^{\psi_0}-\varphi'_0 \phi'_0.
\end{equation}
In the computation in (\ref{eq:2ndterm}), we found
$$
	-(\phi'_0)^2 e^{\psi_0}+\varphi'_0 \phi'_0=(\nabla \phi'_0, \phi'_0\nabla e^{\psi_0}).
$$
Hence, (\ref{eq:remain}) is equivalent to 
\begin{equation}\label{eq:FS_identity}
	\sum_\alpha |(\nabla \phi'_0, \nabla (\sigma^* \tau_\alpha))|_{FS(H(0))}^2
	=
	(\nabla \varphi'_0, \nabla \phi'_0)
	- (\nabla \phi'_0, \phi'_0\nabla e^{\psi_0}).
\end{equation}
This follows from the definition of the restriction $\omega_{FS(H(0))}$ of the Fubini-Study metric.
To see (\ref{eq:FS_identity}), recall that the Fubini-Study metric is given by
$$
	\frac{\sum_{i}dz^i \wedge d\overline{z}^i}{1 + \sum |z^k|^2}
	- \frac{(\sum \overline{z}^i dz^i) \wedge (\sum z^j d \overline{z}^j)}{1 + \sum |z^k|^2}
$$
on the coordinate chart $U_0= \{(1, z^2, \ldots, z^{N}) \in \mathbb{C}P^{N-1}\}$.
Then, we have
\begin{equation}\label{eq:FS_identity_left}
	|(\nabla \phi'_0, \nabla \tau_\alpha)|_{FS(H(0))}^2
	=
	\frac{(\lambda_\alpha^2  +(\phi'_0)^2 -2 \phi'_0 \lambda_\alpha)|\tau_{\alpha}|^2}
	{\sum_\beta |\tau_\beta|^2},
\end{equation}
\begin{equation}\label{eq:FS_identity_right1}
	(\nabla \phi'_0, \nabla |\tau_\alpha|^2_{FS(H(0))})
	=
	\frac{(\lambda_\alpha  -\phi'_0 )|\tau_{\alpha}|^2}
	{\sum_\beta |\tau_\beta|^2},
\end{equation}
and
\begin{equation}\label{eq:FS_identity_right2}
	(\nabla \phi'_0, \phi'_0 \nabla |\tau_\alpha|_{FS(H(0))}^2)
	=
	\frac{\phi'_0 \lambda_\alpha |\tau_\alpha|^2 - (\phi'_0)^2|\tau_\alpha|^2}
	{\sum_{\beta} |\tau_\beta|^2}.
\end{equation}
We get (\ref{eq:FS_identity}) by substituting (\ref{eq:FS_identity_left}) to the left hand of (\ref{eq:FS_identity}), and (\ref{eq:FS_identity_right1}) and (\ref{eq:FS_identity_right2}) to the right hand of (\ref{eq:FS_identity}).
The proof is completed.
\end{proof}
\noindent
The following corollary is fundamental to understand the idea of this paper, although we do not use as it stands in the proof of the main theorem.
\begin{corollary}\label{prop:Z_min}
If $FS_k(H_k)$ is a $\sigma_k$-balanced metric for some $H_k\in \cB_k^G$, then $H_k$ is a minimum of $Z^{\sigma}_k$ on $\cB_k^G$.
\end{corollary}
\begin{proof}
Since $H_k$ is a $\sigma_k$-balanced metric, $\{c(\sigma^* \tau_\alpha)\}_\alpha$ is an orthonormal basis with respect to $T(H_k)$ for some $c>0$.
From (\ref{eq:1st_variation_Z}), $H_k$ is a critical point of $Z^\sigma_k$ on $\cB_k^G$.
From Lemma \ref{lem:Zconvex}, this is an absolute minimum of $Z^\sigma_k$.
\end{proof}

\begin{proposition}
\label{prop:almostbal}
Let $\phi\in\cH^G$. Then there are functions $\ep_{\phi}(k)$ such that
$$
k^{-n}(Z_k^{\sigma}\circ Hilb_k(\phi)) \geq k^{-n}(Z_k^{\sigma}\circ Hilb_k(\phi^*)) + \ep_{\phi}(k)
$$
and such that $\lim_{k \rightarrow \infty} \ep_{\phi}(k)=0$ in  $C^l(X,\R)$ for $l>>1$.
\end{proposition}

\begin{proof}
We follow Li's proof of \cite{li}[Lemma 3.3.], adapted to our more general setting.
In the sequel, $C$ will stand for a constant depending on $\phi$, $\phi^*$ and the volume of the polarized manifold $(X,L)$, but independent on $k$.
The precise value of this constant might change but it won't be important for us.
\par
Let's set $H_k^*=Hilb_k(\phi^*)$ and $H_k=Hilb_k(\phi)$.
We choose an orthonormal basis $\lbrace \tau_{\a}^{(k)} \rbrace$ of $H^*_k$
such that in this basis $H_k^*$ is represented by the identity and
$$
H_k=diag(e^{2\lambda_{\a}^{(k)}}).
$$
Then evaluating $H_k$ on the orthonormal vectors $e^{\lambda_{\a}^{(k)}}\tau_{\a}^{(k)}$:
\begin{equation}
\label{eq1}
e^{-2\lambda_{\a}^{(k)}}=\int_X \vert \tau_{\a}^{(k)}\vert_{h_0^k}^2
d\mu_0.
\end{equation}
Comparing the metrics we have the existence of $C>0$ such that
$$
C^{-k} h_{\phi^*}^k
\leq
h_0^k\leq C^k h_{\phi^*}^k
$$
from which we deduce with (\ref{eq1}) the following estimate:
\begin{equation}
\label{eq:estimelambda}
\vert \lambda_{\a}^{(k)} \vert \leq C k.
\end{equation}
Let's consider the one-parameter subgroup of $\cB_k^G$:
$$
s \mapsto H_k(s)=diag(e^{2s \lambda_{\a}^{(k)}}).
$$
This is a geodesic that goes from $H_k^*$ to $H_k$ in $\cB_k^G$, thus by lemma~\ref{lem:Zconvex}:
$$
k^{-n}(Z_k^{\sigma}(H_k)-Z_k^{\sigma}(H_k^*))\geq k^{-n}f_k'(0)
$$ 
with
$$
f_k(s)=Z_k^{\sigma}(H_k(s)).
$$

\par
We then need to show that $\lim_{k\rightarrow \infty}k^{-n}f_k'(0)=0 $.
By a straightforward computation
$$
k^{-n}f_k'(0)=2k^{-n}\sum_{\a} \lambda_{\a}^{(k)} -\frac{2}{k}\int_X \frac{\rho_k^{\lambda}}{\rho_k} (k+\Delta)e^{\psi_k} d\mu
$$
where $\rho_k^{\lambda}=\sum_{\a}\lambda_{\a}^{(k)} \vert \tau_{\a}^{(k)}\vert^2_{h_0^k}$ and the quantities $\rho_k$, $\Delta$, $\psi_k$ and $d\mu$ are computed with respect to the extremal metric $\om_{\phi^*}$.
Then
\begin{equation}
\label{eq:fprime}
2^{-1}k^{-n}f_k'(0)=k^{-n}\sum_{\a} \lambda_{\a}^{(k)}-\int_X \frac{\rho_k^{\lambda}}{\rho_k} e^{\psi_k} d\mu - \frac{1}{k} \int_X  \frac{\rho_k^{\lambda}}{\rho_k} \Delta e^{\psi_k} d\mu.
\end{equation}

\noindent We first show that the last term in the sum of (\ref{eq:fprime}) tends to zero.
First note that from (\ref{eq:estimelambda}), 
$$
\vert\frac{\rho_k^{\lambda}}{\rho_k}\vert \leq Ck
$$
thus
$$
\vert\frac{1}{k} \int_X  \frac{\rho_k^{\lambda}}{\rho_k} \Delta e^{\psi_k} d\mu \vert\leq 
C  \int_X \vert\Delta e^{\psi_k}\vert d\mu
$$
and using lemma \ref{lem:exppsi} we deduce that this term goes to zero as
$k$ tends to infinity.

\noindent Then consider the second term in the right hand side of equation (\ref{eq:fprime}).
Using the expansions of $\psi_k$ and $\rho_k$ we deduce:
$$
\rho_k^{-1} e^{\psi_k}=k^{-n}(1-\frac{S}{2k}+\cO(k^{-2}))(1+\frac{\theta+\uS}{2k}+\co(k^{-1})).
$$
Here we use our crucial assumption,
that is $\om_{\phi^*}$ is extremal, so $S=\theta + \uS$ and thus
$$
\rho_k^{-1} e^{\psi_k}=k^{-n}(1+\co(k^{-1})).
$$
Then
$$
\int_X \frac{\rho_k^{\lambda}}{\rho_k} e^{\psi_k} d\mu =\int_X \frac{\rho_k^{\lambda}}{k^n}(1+\co(k^{-1})) d\mu.
$$
As
$$
\int_X \frac{\rho_k^{\lambda}}{k^n} d\mu= k^{-n}\sum_{\a} \lambda_{\a}^{(k)},
$$
the only remaining term to control at infinity in $k^{-n}f_k'(0)$ is
$$
\int_X \frac{\rho_k^{\lambda}}{k^n}\co(k^{-1}) d\mu.
$$
Using (\ref{eq:estimelambda}),
$$
\vert \frac{\rho_k^{\lambda}}{k^n}\co(k^{-1}) \vert \leq C k N_k k^{-n} \vert \co(k^{-1}) \vert.
$$
By equation (\ref{cor:Nk}), $N_k k^{-n}$ is bounded and as $\co(k^{-1})=k^{-1}\epsilon(k)$ with $\epsilon(k)\rightarrow 0$
$$
\lim_{k\rightarrow \infty}\int_X \frac{\rho_k^{\lambda}}{k^n}\co(k^{-1}) d\mu =0
$$
and
$$
\lim_{k\rightarrow \infty} k^{-n}f_k'(0)=0.
$$
\end{proof}

\subsection{Conclusion, proof of theorem \ref{theo:min}}

We conclude this section with the proof of Theorem~\ref{theo:min}. We show the following stronger theorem, which implies theorem \ref{theo:min} with remark \ref{rmk:mab}:

\begin{theorem}
Let $(X,L)$ be a polarized manifold that carries extremal metrics representing $c_1(L)$.
The modified Mabuchi functional with respect to the $G$-action induced by the extremal vector field of $c_1(L)$ attains its minimum at the extremal metrics.
\end{theorem}

\begin{proof}
Let $\phi\in\cH^G$ and $\phi^*$ be the potential of an extremal metric.
\begin{equation}
\label{eq:end}
\begin{array}{ccc}
\cL_k^{\sigma}(\phi) & = & Z_k^{\sigma}\circ Hilb_k(\phi)+(\cL_k^{\sigma}(\phi)-Z_k^{\sigma}\circ Hilb_k(\phi)).
\end{array}
\end{equation}
By proposition~\ref{prop:almostbal}:
\begin{equation}
\label{eq:end}
\begin{array}{ccc}
\cL_k^{\sigma}(\phi)  & \geq & Z_k^{\sigma}\circ Hilb_k(\phi^*)+ k^n\ep_{\phi}(k)+(\cL_k^{\sigma}(\phi)-Z_k^{\sigma}\circ Hilb_k(\phi)) \\
\end{array}
\end{equation}
Then
\begin{equation}
\label{eq:end}
\begin{array}{ccc}
\cL_k^{\sigma}(\phi)   & \geq & \cL_k^{\sigma}(\phi^*)+(Z_k^{\sigma}\circ Hilb_k(\phi^*)-\cL_k^{\sigma}(\phi^*))+ \\
 & & k^n\ep_{\phi}(k)+(\cL_k^{\sigma}(\phi)-Z_k^{\sigma}\circ Hilb_k(\phi))\\
\end{array}
\end{equation}
To conclude, from proposition~\ref{prop:compare}, 
$$
k^{-n}(Z_k^{\sigma}\circ Hilb_k(\phi^*)-\cL_k^{\sigma}(\phi^*)) \rightarrow 0
$$
and
$$
k^{-n}(Z_k^{\sigma}\circ Hilb_k(\phi)-\cL_k^{\sigma}(\phi)) \rightarrow 0
$$ 
as $k$ tends to infinity. So does $\ep_{\phi}(k)$ by construction, see proposition~\ref{prop:almostbal}. Thus the result follows from proposition~\ref{prop:LquantizeE}, multiplying by $k^{-n}$ and letting $k$ go to infinity in~(\ref{eq:end}).
\end{proof}

%%%Appendix by Sano (beginning)%%%
\section{Appendix}
We give the proof of the results concerning the $\sigma$-balanced metrics.
We denote by $(\cdot, \cdot)$ any of the following Hermitian pairings
\begin{equation*}
	\begin{array}{ll}
		T^*X \times (T^*X\times L) \to L, & L\times (T^*X\times L) \to T^*X, 
	\\
		L\times L \to \mathbb{C}, & T^*X \times T^*X \to \mathbb{C}
	\end{array}
\end{equation*}
obtained by $\phi\in \mathcal{H}$ and $\omega_\phi$.
We denote the connection of type $(1,0)$ on the holomorphic tangent bundle $T'X$  by $\nabla$.
\subsection{The definition of $I^{\sigma}$}
\begin{proposition}
\label{prop:indep}
$I^\sigma(\phi)$ is independent of the choice of a path from $0$ to $\phi$.
\end{proposition}
\begin{proof}
Since $I^\sigma(\phi)$ satisfies the cocycle property
$$
	I^\sigma(\phi_1,\phi_3)
	=
	I^\sigma(\phi_1,\phi_2)
	+
	I^\sigma(\phi_2,\phi_3)
$$
by definition, it is sufficient to prove $\frac{\partial^2}{\partial s \partial t}I^\sigma(\phi_{0,0},\phi_{t,s})$ is symmetric with respect to $s$ and $t$ for any family of path 
$$
	\{\Phi=\phi_{t,s} \mid (s,t)\in [0,1]\times [0,1],\,  \phi_{0,s}=\phi_{1,s}\equiv 0\}
$$
in $\mathcal{H}$.
\begin{eqnarray}
	\nonumber
	&&
		\frac{\partial^2}{\partial s \partial t}I^\sigma(\phi_{0,0},\phi_{t,s})
		=
		\frac{\partial}{\partial s}\int_X \big((1+\Delta_\Phi)\frac{\partial \Phi}{\partial t}\big)e^{\psi_{\sigma,\Phi}}d\mu_{\Phi}
	\\
	\nonumber
	&=&
		\int_X \big((\frac{\partial}{\partial s}\Delta_\Phi)\frac{\partial\Phi}{\partial t}\big)e^{\psi_{\sigma,\Phi}}
		d\mu_{\Phi}
		+\int_X \big((1+\Delta_\Phi)\frac{\partial^2\Phi}{\partial s \partial t}\big)e^{\psi_{\sigma,\Phi}}
		d\mu_{\Phi}
	\\
	\label{eq:dtds_F0}
	&&\quad
		+\int_X \big((1+\Delta_\Phi)\frac{\partial\Phi}{\partial t}\big)\big(\frac{\partial e^{\psi_{\sigma,\Phi}}}{\partial s}\big)d\mu_{\Phi}
		-\int_X \big((1+\Delta_\Phi)\frac{\partial\Phi}{\partial t}\big)e^{\psi_{\sigma,\Phi}}\big(\Delta_\Phi\frac{\partial\Phi}{\partial s}\big)d\mu_{\Phi}.
\end{eqnarray}
The first term in (\ref{eq:dtds_F0}) is
\begin{eqnarray*}
		\int_X \big(\nabla\overline{\nabla}\frac{\partial\Phi}{\partial t},\nabla\overline{\nabla} \frac{\partial\Phi}{\partial s}\big)e^{\psi_{\sigma,\Phi}}
		d\mu_{\Phi}
\end{eqnarray*}
which is symmetric.
The second term is obviously symmetric.
The third term is
\begin{equation}\label{eq:dd_F0_3}
		\int_X \frac{\partial\Phi}{\partial t}\big(\nabla \psi_{\sigma,\Phi}, \nabla \frac{\partial\Phi}{\partial s}\big)e^{\psi_{\sigma,\Phi}}
		d\mu_{\Phi}
		+
		\int_X \big(\Delta_\Phi\frac{\partial\Phi}{\partial t}\big)\big(\nabla\psi_{\sigma,\Phi}, \nabla \frac{\partial\Phi}{\partial s}\big)e^{\psi_{\sigma,\Phi}}
		d\mu_{\Phi}.
\end{equation}
Here we use the following equality.
\begin{lemma}
\begin{equation}\label{eq:derivative_psi}
	\frac{\partial\psi_{\sigma,\Phi}}{\partial s}=\big(\nabla\psi_{\sigma,\Phi},\nabla\frac{\partial\Phi}{\partial s}\big).
\end{equation}
\end{lemma}
\begin{proof}
Let $v$ be the gradient vector field of $\frac{\partial\Phi}{\partial s}$, i.e.,
\begin{equation}\label{eq:cf}
	v
	=grad_{\omega_{\Phi}}\bigg(\frac{\partial\Phi}{\partial s}\bigg)
	=\sum_{i,j}g^{i\bar{j}}\frac{\partial}{\partial\bar z^j}\bigg(\frac{\partial\Phi}{\partial s}\bigg)
	\frac{\partial}{\partial z^i}.
\end{equation}
We have
\begin{eqnarray*}
		\frac{\partial}{\partial s}(\sigma(1)^*\omega_{\Phi}-\omega_{\Phi})
	&=&
		L_{v}(\sigma(1)^*\omega_{\Phi}-\omega_{\Phi})
		=
		\frac{\sqrt{-1}}{2\pi}
		d\iota_{v}
		\partial\bar\partial\psi_{\sigma,\Phi}
	\\
	&=&
		\frac{\sqrt{-1}}{2\pi}
		\partial\bar\partial \big(\nabla\psi_{\sigma,\Phi},\nabla\frac{\partial\Phi}{\partial s}\big)
\end{eqnarray*}
where $L_{v}$ is the Lie derivative along $v$.
Then, there exists some constant $c$ such that
\begin{equation}
	\label{eq:constant_c}
		\frac{\partial\psi_{\sigma,\Phi}}{\partial s}
	=
		\big(\nabla\psi_{\sigma,\Phi},\nabla\frac{\partial\Phi}{\partial s}\big)+c.
\end{equation}
Recall that 
$$
	\int_X\psi_{\sigma,\Phi}d\mu_{\Phi}
$$
is constant with respect to $s,\,t$ by normalization of $\psi_{\sigma,\Phi}$.
Since
$$
	0
	=\frac{\partial}{\partial s}\int_X {\psi_{\sigma,\Phi}}d\mu_{\Phi}
	=\int_X \bigg(\frac{\partial\psi_{\sigma,\Phi}}{\partial s}-\big(\nabla\psi_{\sigma,\Phi},\nabla\frac{\partial\Phi}{\partial s}\big)\bigg)
	d\mu_{\Phi},
$$
the constant $c$ in (\ref{eq:constant_c}) is zero.
Hence, (\ref{eq:derivative_psi}) is proved.
\end{proof}
\noindent
The forth term is
\begin{equation}\label{eq:dd_F0_4}
	-\int_X e^{\psi_{\sigma,\Phi}}\frac{\partial\Phi}{\partial t}\Delta_\Phi\frac{\partial\Phi}{\partial s}d\mu_{\Phi}
	-
	\int_X e^{\psi_{\sigma,\Phi}}\Delta_\Phi\frac{\partial\Phi}{\partial t}\Delta_\Phi\frac{\partial\Phi}{\partial s}d\mu_{\Phi}.
\end{equation}
The sum of the first term in (\ref{eq:dd_F0_3}) and the first term in (\ref{eq:dd_F0_4}) is
$$
	-\int_X \frac{\partial\Phi}{\partial t}\bigg(\Delta_\Phi\frac{\partial\Phi}{\partial s}+\big(\nabla\psi_{\sigma,\Phi}, \nabla\frac{\partial\Phi}{\partial s}\big)\bigg)e^{\psi_{\sigma,\Phi}}
	d\mu_{\Phi}.
$$
This is symmetric, because the operator $\Delta_\Phi+(\nabla\psi_{\sigma,\Phi}, \nabla )$ is self-adjoint with respect to the weighted volume form $e^{\psi_{\sigma,\Phi}}d\mu_{\Phi}$.
The remaining is the second term in (\ref{eq:dd_F0_3}).
It is 
$$
		-\int_X \big(\nabla\overline{\nabla}\psi_{\sigma,\Phi},
			\nabla\frac{\partial\Phi}{\partial t}\overline{\nabla}\frac{\partial\Phi}{\partial s}
		\big)e^{\psi_{\sigma,\Phi}}
		d\mu_{\Phi}
	-
		\int_X\big(\nabla\frac{\partial\Phi}{\partial t},\nabla\psi_{\sigma,\Phi}\big)
		\big(\nabla\frac{\partial\Phi}{\partial s},\nabla\psi_{\sigma,\Phi}\big)
		e^{\psi_{\sigma,\Phi}}
		d\mu_{\Phi},
$$
which is symmetric.
\end{proof}
\subsection{Second derivative of $I_k^{\sigma}$}
\label{sec:hessian}

We give a computation of the second derivative of $I^\sigma_k$.

\begin{proof}[Proof of Lemma \ref{lem:hessian}]
\begin{eqnarray}
	\nonumber
		\frac{d^2 }{ds^2}I^\sigma_k(\phi_s)
	&=&
		k^n\frac{d}{ds}\int_X (k+\Delta_\phi )\phi' e^{{\psi_{\sigma,\phi}}} d\mu_{\phi}
	\\
	\nonumber
	&=&
		k^n\int_X (\nabla\overline{\nabla}\phi',\nabla\overline{\nabla}\phi')e^{\psi_{\sigma,\phi}}
		 d\mu_{\phi}
		+k^n\int_X(k+\Delta_\phi )\phi''e^{\psi_{\sigma,\phi}} d\mu_{\phi}
	\\
	\label{eq:2nd_vr_01}
	&&\quad
		+k^n\int_X((k+\Delta_\phi )\phi'){\psi'_{\sigma,\phi}}e^{\psi_{\sigma,\phi}} d\mu_{\phi}
		-k^n\int_X((k+\Delta_\phi )\phi')e^{\psi_{\sigma,\phi}}\Delta_\phi \phi' d\mu_{\phi}.
\end{eqnarray}
From (\ref{eq:derivative_psi}), the third term in (\ref{eq:2nd_vr_01}) is equal to
\begin{equation}\label{eq:2nd_vr_02}
	k^n\int_X((k+\Delta_\phi )\phi')(\nabla{\psi_{\sigma,\phi}},\nabla\phi')e^{\psi_{\sigma,\phi}} d\mu_{\phi}.
\end{equation}
By the partial integral, the forth term in (\ref{eq:2nd_vr_01}) is equal to
\begin{eqnarray}
	\nonumber
	&&
			-k^{n+1}\int_X|\nabla\phi'|^2e^{\psi_{\sigma,\phi}} d\mu_{\phi}
			-k^{n+1}\int_X\phi'e^{\psi_{\sigma,\phi}}(\nabla{\psi_{\sigma,\phi}},\nabla\phi') d\mu_{\phi}
	\\
	\label{eq:2nd_vr_03}
	&&\quad
		-k^n\int_X(\nabla\Delta_\phi \phi',\nabla\phi')e^{\psi_{\sigma,\phi}} d\mu_{\phi}
		-k^n\int_X(\Delta_\phi \phi')	(\nabla{\psi_{\sigma,\phi}},\nabla\phi')e^{\psi_{\sigma,\phi}} d\mu_{\phi}.
\end{eqnarray}
Remark that the sum of the second and forth terms in (\ref{eq:2nd_vr_03}) cancels (\ref{eq:2nd_vr_02}).
The third term in (\ref{eq:2nd_vr_03}) is
\begin{eqnarray}
	\nonumber
	&&
		-k^n\int_X
		(\nabla\overline{\nabla}\phi',\nabla\overline{\nabla}\phi')
		e^{\psi_{\sigma,\phi}} d\mu_{\phi}
		-k^n\int_X
		(\nabla\overline{\nabla}\phi',\nabla{\psi_{\sigma,\phi}}\overline{\nabla}\phi')
		e^{\psi_{\sigma,\phi}} d\mu_{\phi}
	\\
	\nonumber
	&=&
		-k^n\int_X
		(\nabla\overline{\nabla}\phi',\nabla\overline{\nabla}\phi')
		e^{\psi_{\sigma,\phi}} d\mu_{\phi}
		-k^n\int_X
		|\nabla\phi'|^2\Delta_\phi {\psi_{\sigma,\phi}} 
		e^{\psi_{\sigma,\phi}} d\mu_{\phi}
	\\
	\nonumber
	&&\quad
		+k^n\int_X
		|\nabla\phi'|^2|\nabla{\psi_{\sigma,\phi}}|^2
		e^{\psi_{\sigma,\phi}} d\mu_{\phi}
	\\
	\label{eq:2nd_vr_04}
	&=&
		-k^n\int_X
		(\nabla\overline{\nabla}\phi',\nabla\overline{\nabla}\phi')
		e^{\psi_{\sigma,\phi}} d\mu_{\phi}
		-k^n\int_X
		|\nabla\phi'|^2\Delta_\phi 
		e^{\psi_{\sigma,\phi}} d\mu_{\phi}.		
\end{eqnarray}
Substituting (\ref{eq:2nd_vr_02}), (\ref{eq:2nd_vr_03}) and (\ref{eq:2nd_vr_04}) for (\ref{eq:2nd_vr_01}), we get the second derivative of $I^\sigma_k(\phi)$.
\end{proof}

%%%Appendix by Sano (end)%%%

\end{document}